\documentclass{amsart}
\usepackage{amsmath,amsthm, amssymb}
\usepackage{graphicx}
\usepackage{graphics}
\usepackage{esint} 
\usepackage{mathrsfs}
\usepackage{mathtools}

\def\Xint#1{\mathchoice
   {\XXint\displaystyle\textstyle{#1}}%
   {\XXint\textstyle\scriptstyle{#1}}%
   {\XXint\scriptstyle\scriptscriptstyle{#1}}%
   {\XXint\scriptscriptstyle\scriptscriptstyle{#1}}%
   \!\int}
\def\XXint#1#2#3{{\setbox0=\hbox{$#1{#2#3}{\int}$}
     \vcenter{\hbox{$#2#3$}}\kern-.6\wd0}}
\def\dashint{\Xint-}

\newcommand{\R}{{\mathbb R}}

\newcommand{\CC}{{\mathfrak C}}

\newcommand{\UU}{{\mathscr U}}
\newcommand{\A}{{\mathcal  A}}
\newcommand{\B}{{\mathcal  B}}
\newcommand{\D}{{\mathcal  D}}
\newcommand{\F}{{\mathcal  F}}
\newcommand{\I}{{\mathcal  I}}

\newcommand{\U}{{\mathcal U}}

\newcommand{\Q}{{\mathcal Q}}

\newcommand{\loc}{{\rm loc}}

\newcommand{\osc}{{\rm osc\,}}
\newcommand{\esssup}{\mathop{\rm ess\,sup}\limits}

\numberwithin{equation}{section}

\def\Xint#1{\mathchoice
    {\XXint\displaystyle\textstyle{#1}}%
     {\XXint\textstyle\scriptstyle{#1}}%
     {\XXint\scriptstyle\scriptscriptstyle{#1}}%
     {\XXint\scriptstyle\scriptscriptstyle{#1}}%
    \!\int}
\def\XXint#1#2#3{{\setbox0=\hbox{$#1{#2#3}{\int}$}
    \vcenter{\hbox{$#2#3$}}\kern-.5\wd0}}

 \newtheorem{thm}{Theorem}[section]
 
 \newtheorem{lem}[thm]{Lemma}
 
 \theoremstyle{definition}
 \newtheorem{defn}[thm]{Definition}
 \theoremstyle{remark}

 \numberwithin{equation}{section}

\begin{document}

\title[Quasilinear Parabolic  Problems with Perturbed Coefficients]{Approximation of the Solutions\\ to  Quasilinear Parabolic Problems\\ with Perturbed  $VMO_x$ Coefficients}

\author[R. Rescigno]{Rosamaria Rescigno}


\author[L.G.  Softova]{Lubomira G. Softova}


\date{}

\begin{abstract}
We consider the Cauchy-Dirichlet problem for second-order quasilinear  operators of parabolic type in non-divergence form. The data are Cara\-thé\-o\-dory functions, and the principal part is of $VMO_x$-type with respect to the variables $ (x,t).$  Assuming the existence of a strong solution $u_0,$  we apply the Implicit Function Theorem in a neighbourhood  of this solution to show that small bounded perturbations of the data  lead to small perturbations of the solution $u_0$ itself. Furthermore, we employ the Newton Iteration Procedure to construct an approximating sequence that converges  to  $u_0$ in the corresponding Sobolev space.
\end{abstract}

\maketitle

\section{Introduction}
Let $\Omega \subset \mathbb{R}^n$, $n \geq 2$, be a bounded domain with a $C^{1,1}$-smooth boundary $\partial \Omega$. Define the space-time cylinder  $Q = \Omega \times (0, T)$ for some  $T > 0$,  and let  
$$
\partial_p Q = \big(\partial \Omega \times (0, T)\big) \cup \big(\overline{\Omega} \times \{0\}\big)
$$
denote the parabolic boundary of $Q$.  

We consider the following Cauchy-Dirichlet problem for a second-order quasilinear parabolic equation in non-divergence form:  
\begin{equation} \label{PD}
\begin{cases}
D_t u - a^{ij}(x, t, u, Du) D_{ij} u = f(x, t, u, Du), & \text{a.e. in } Q, \\
u(x, t) = 0, & \text{on } \partial_p Q,
\end{cases}
\end{equation}  
where the  summation convention over  repeated indices is used.  

The functions  $a(x,t,u,\xi)$ and $f(x,t,u,\xi)$ are assumed  to satisfy the {\it Carathéodory conditions}: 
 these are continuous with respect to $(u, \xi) \in \mathbb{R} \times \mathbb{R}^n$ for almost all  $(x,t) \in Q$ and
measurable in $(x,t)$ for all $(u, \xi) \in \mathbb{R} \times \mathbb{R}^n.$  

The maximal regularity theory in Sobolev spaces for linear elliptic and parabolic equations with discontinuous coefficients (see for instance   \cite{CMP,CFL,DK,Pl,Re,Sf} and the references therein), combined with linearization techniques and the \textit{Implicit Function Theorem}, provides a powerful framework for treating  nonlinear problems with perturbed data. These perturbations typically involve functions with  small $L^\infty$ norms (cf.  \cite{GRk,GRk1,PRS,Rk,RS}).

Recently   Krylov \cite{Kr} studied linear parabolic equations with coefficients having  {\it Vanishing 
Mean Oscillation} $(VMO)$ in the spatial variable  $x$ and merely measurable dependence on  the time $t$ and proved   existence of unique  strong solution in the strip $\Omega(T)=\R^n\times(0,T)$ (see also \cite{KimKr,Kr1}).
Later,  Dong and Kim  \cite{DK} extended  Krylov's results  to  higher order elliptic and parabolic systems in various  unbounded and bounded domains with Dirichlet data prescribed on the lateral boundary. 
Further  advancements of this theory can be found  in \cite{ANPS1,ANPS3}, where the authors treat linear and quasilinear problems for parabolic equations with partially $VMO$ coefficients and Venttsel boundary conditions.

This theory  serves as a foundation for the analysis of   nonlinear parabolic  equations  with  discontinuous principal parts of $VMO_x$-type  that are  
additionally perturbed by  small $L^\infty$ functions.
Another kind    of proximity condition,  introduced by  Campanato  \cite{CMP,CMP1}, permits  the  study of nonlinear operators in non-divergence form  that are {\it "near"} to   well-studied linear operators  such  as the Laplacian or  the heat operator  (see also \cite{CMP0,Sf1,Sf2} and the references therein).  The {"nearness condition"} of Campanato has been used by Tarsia  \cite{TRS} to generalize 
 the Implicit Function Theorem. In that paper, the hypothesis of bijectivity of the Frèchet differential  is replaced by the assumption of {\it nearness between the function and an open and injective operator.} This allowed   the author to  prove unique strong solvability of the Dirichlet problem for nonlinear non-variational  systems. 
Other approximation results were  obtained by \ Byun, Wang, and their collaborators   (see \cite{BJW,BS} and the references therein), establishing    local existence, uniqueness, and smooth dependence of solutions on the data for broad classes of nonlinear  equations that are   {\it "asymptotically close to a regular  operator"}  such as the $p$-Laplacian. 

With these analytical results at hand,
  the classical \textit{Newton Iteration Procedure}, known for  its quadratic convergence, provides a constructive approach  to generate  approximating sequences converging to a strong  solution.

 The general goal of this paper  is to show that, under suitable assumptions such as the well-posedness of the Cauchy-Dirichlet problem~\eqref{PD} and the regularity required on the data, the underlying differential operator defines a \textit{Fredholm mapping of index zero}. This property  together with the maximum principle  \cite{N} guarantee that the operator is an \textit{isomorphism} from $  W^{2,1}_p(Q) $  onto $  L^p(Q) $  for all  $p > n + 2.$  
This framework enables a rigorous analysis of the effects of small perturbations in the data. 
Namely, we prove that given a strong solution
$u_0$ of \eqref{PD}, then   small $L^\infty$-perturbations of the coefficients and the source term  lead to a unique perturbed solution $u,$ close to $u_0$ in the  $W^{2,1}_p(Q)$ norm, which depends continuously on the perturbations.

The results obtained here  are particularly relevant for both analytical and numerical investigations of quasilinear parabolic differential equations with discontinuous  coefficients, such as:  
\begin{itemize}
    \item establishing existence and uniqueness of solutions under perturbations;
    \item providing constructive approximation methods via iterative techniques;
    \item ensuring stability of solutions with respect to small data variations.
\end{itemize}  
They have direct  application  in areas   such as   diffusion  in heterogeneous media, nonlinear heat conduction, and incompressible fluid flow, where quasilinear parabolic PDEs naturally arise.

We conclude this section with some  notation used throughout the paper:
\begin{itemize}
    \item $x=(x_1,\dots,x_n)\in \R^n, $  with Euclidean norm   $|x|;$  
     \item $\mathcal{B}_r(x)=\{y\in \R^n: \ |x-y|<r\}$ denotes  an open  ball in $\R^n$ and for any bounded domain $\Omega\subset \R^n$ we write   $\Omega_r(x)=\Omega\cap \B_r(x)$ with $x\in \Omega$ and $r>0;$ 
    \item $
    \mathcal{I}_r(x,t)= \mathcal{B}_r(x)\times (t,t+r^2)$
    is   a \textit{parabolic cylinder} in $\R^n\times\R^+$ and we write
    $\Q_r(x,t)=Q\cap \I_r(x,t)$ with $(x,t)\in Q.$ For any $\tau\in (0,T)$ we denote  $Q_\tau=\Omega\times (0,\tau);$
    
    \item For a function $u:\R^n\times \R\to \R,$  the partial derivatives are denoted by
    $
    D_i u = \partial u/\partial x_i, $  $ D_t u = \partial u/\partial t, $ $ Du = (D_1 u, \dots, D_n u), $  $  D^2 u = \{D_{ij} u\}_{i,j=1}^n. $
    \item For a function $a(x,t,u,\xi):Q\times \R\times \R^n\to \R,$ we  write $a_u$ ad $D_{\xi_l}a$ for the partial derivatives with respect to $u$ and the $l$-th component of $\xi\in \R^n.$
   
    \item By $L^p(Q), W^{2,1}_p(Q), p\in (1,\infty),$ and $W^{1,\infty}_x(Q)$ we denote the classical parabolic Lebesgue and Sobolev spaces with the corresponding norms
    \begin{align*}      
    &\|u\|^p_{p,Q}=\int_{Q} |u(x,t)|^p\,dxdt, \qquad \|u\|_{\infty,Q}=\esssup_{(x,t)\in Q} |u(x,t)|,\\
&\|u\|_{W^{2,1}_p(Q)}=\|u\|_{p,Q}+\|Du\|_{p,Q}+\|D^2u\|_{p,Q}+\|D_tu\|_{p,Q},\\
    &\|u\|_{W_x^{1,\infty}(Q)}= \|u\|_{\infty,Q}+\|D u\|_{\infty,Q};
     \end{align*}
    
    \item  The {\it mean integral} of a function $u$ on $Q$ is denoted  by
    $$
    \Xint-_{Q} u(x,t)\,dxdt=\frac{1}{|Q|}\int_0^T\int_\Omega u(x,t)\,dxdt.
    $$
    
\end{itemize}

\section{Definitions and Auxiliary Results}

To describe the regularity of the data, we introduce the following classes of functions. 
\begin{defn}\label{del-C11}
Let $\mathcal U\subseteq \R\times\R^n,$ and let  $a:Q\times \mathcal U\to \R$ be a Carathéodory function.
\begin{enumerate}
\item 
We say that $a$ is  a $\CC^1(Q\times \U)$-Carathéodory function if: 
\begin{itemize}
\item[($A_1$)] $a(x,t,u,\xi)$ is  continuously differentiable  in $(u,\xi)\in \mathcal U$ for almost every  $ (x,t)\in Q$ and the functions  $a, a_u,$ and $ D_{\xi_l}a$ are Carath\'eodory.  
\item[($A_2$)] The functions $a, a_u,$ and $ D_{\xi_l}a$ are locally uniformly bounded:   for every compact $K\subset \U $   there exists a constant $C_K>0$ such that  
$$
\|a(u,\xi)\|_{\infty,Q} + \|a_u(u,\xi)\|_{\infty,Q} + \sum_{l=1}^n\|D_{\xi_l} a(u,\xi)\|_{\infty,Q} \leq C_K
$$
for  all $(u,\xi)\in K,$ where 
$
\|a(u,\xi)\|_{\infty,Q} := \esssup_{(x,t)\in Q} |a(x,t,u,\xi)|.
$
\item[($A_3$)]  $a, a_u,$ and $ D_{\xi_l}a$ are  locally  uniformly continuous in $(u,\xi)$: 
for  every  $\varepsilon>0$ and any compact $K\subset \mathcal U,$ there exists $\delta=\delta(\varepsilon,K)>0$ such that if   $(u,\xi), (\overline{u},\overline{\xi})\in K$ with $|u-\overline{u}|+|\xi-\overline{\xi}|< \delta ,
$  then
\begin{align*}
    \|a(u,\xi)-a(\overline{u},\overline{\xi})\|_{\infty,Q} + &\|a_u(u,\xi)-a_u(\overline{u},\overline{\xi})\|_{\infty,Q}\\[5pt]
    +&\sum_{l=1}^{n} \|D_{\xi_l}a(u,\xi)-D_{\xi_l}a(\overline{u},\overline{\xi})\|_{\infty,Q} < \varepsilon.
\end{align*}

\end{itemize}

\item  For each compact $K\subset \U $, define the norm
\begin{equation*}
\begin{split}
\|a\|_{\CC^1(Q\times K)}:= \sup_{(u,\xi)\in K} &\Big(\|a(u,\xi)\|_{\infty,Q}\\[5pt]
&+\|a_u(u,\xi)\|_{\infty,Q}+\sum_{l=1}^n \|D_{\xi_l}a(u,\xi)\|_{\infty,Q}\Big).
\end{split}
\end{equation*}
\item We say that  $a$ is  a $\CC^{1,1}(Q\times  \mathcal U)$-Carathéodory function if $a\in \CC^1(Q\times\U)$ and $a, a_u,$ and $ D_{\xi_l}a$ are locally Lipschitz in  $(u,\xi)\in \U.$ That is, for every  compact $K\subset 
    \mathcal{U} ,$  there exists a constant $L=L(a,K)>0$ such that
\begin{align*}
\|a(u,\xi)-a(\overline{u},\overline{\xi})\|_{\infty,Q}
&+\|a_u(u,\xi)-a_u(\overline{u},\overline{\xi})\|_{\infty,Q}\\
&+\sum_{l=1}^n \|D_{\xi_l}a(u,\xi)-D_{\xi_l}a(\overline{u},\overline{\xi})\|_{\infty,Q}\\
&\leq L (|u-\overline{u}|+|\xi-\overline{\xi}|)
\end{align*}
for  all $(u,\xi),(\overline{u},\overline{\xi})\in K.$
\end{enumerate}
\end{defn}

The next class of functions with partially $VMO$ regularity was  introduced by Krilov, who studied  parabolic equations  with coefficients that  are  measurable in time  and have  {\it Vanishing Mean Oscillation}  with  respect to the spatial variable.  These ideas have been used in  \cite{Re} to investigate the strong solvability of  the Cauchy-Dirichlet problem for   quasilinear parabolic equations   under  Krylov-type  regularity assumptions  on the coefficients.

\begin{defn}\label{VMO}
Let  $a\in L^1_\loc(\R^n\times \R).$ We define the \textit{mean oscillation} of $a$ in the spatial variable   $x$ over a parabolic cylinder   $\mathcal{I}_r(x,t)$ by  
$$
   \osc_x(a,\mathcal{I}_r(x,t)):=  \Xint-_{\mathcal{I}_r(x,t)} \Xint-_{\mathcal{B}_r(x)} |a(y ,\tau )- a(z ,\tau )|\,dydzd\tau,
$$
and introduce the function
\begin{equation}\label{eq-modulus}
    a^{\#(x)}(R)=  \sup_{r\leq R}\, \sup_{(x,t)\in\R^{n+1}}  \,\osc_x(a,\mathcal{I}_r(x,t)).
\end{equation}
We say that $a$ belongs to $ VMO_x$ with a $VMO_x$-modulus given by  the quantity  \eqref{eq-modulus}, if 
$$
\lim_{R\to0} a^{\#(x)}(R) =0.
$$
\end{defn}

Our goal is to investigate how   solutions of \eqref{PD} depend on small perturbations of the data.
 To this end,  we introduce   the appropriate functional framework for
the solutions.

In what follows we always consider the exponent  
$p\in (n+2,\infty).$  Define the function space 
\begin{equation}\label{eq-Wp}
W_p(Q)=\Big\{\,u\in W^{2,1}_p(Q), \qquad
u(x,t)=0 \text{ on }\partial_p Q\,\Big\}
\end{equation}
endowed with the  standard parabolic   Sobolev norm (cf.  \cite{LSU}). 
It is worth to note that, by   the embedding properties of parabolic Sobolev spaces (cf. \cite[Lemma 3.3]{LSU}), the   functions in $W_p(Q)$  are  continuous up to the  parabolic boundary, together with their spatial gradients.  
 Hence $Du\in VMO_x\cap L^\infty(Q;\R^n)$ automatically
and  we denote  the $VMO_x$-modulus of the gradient $Du$ by 
$$
Du^{\#(x)}(R):=\sum_{i=1}^n (D_iu)^{\#(x)}(R).
$$

The following lemmas describe the regularity  of the composition operator associated with $a$ and will be central to the subsequent analysis. They are analogous  to  \cite[Lemma~1]{PRS} and \cite[Lemma~3.2]{RS}.
 \begin{lem} \label{lm1}
 Let $a$  be a Carath\'eodry function satisfying  the following conditions:
 \begin{itemize}
     \item[($C_1$)]  For every $(u,\xi)\in \U,$ the function $(x,t)\to a(x,t, u, \xi)$ belongs to $VMO_x\cap L^\infty(Q),$ locally uniformly in $(u,\xi),$ with   $VMO_x$-modulus given by 
\begin{equation}\label{VMOxmod}
a^{\#(x)}_{M}(R) := \sup_{|u|, |\xi| \leq M} \sup_{\substack{(x,t) \in Q\\ r\leq R}} 
\Xint-_{\Q_r(x,t)} \Xint-_{\Omega_r} 
|a(y,\tau,u,\xi) - a(z,\tau,u,\xi)| \, dz  dy  d\tau
\end{equation}
with some positive constant $M.$
     
     \item[($C_2$)]  $a(x,t,u,\xi)$ is locally uniformly continuous in $(u,\xi)$. That is, for each compact   $K\subset   \U,$ there exist a constant $C_K>0$ and a non decreasing  function 
     $$
     \mu_K:(0,+\infty)\rightarrow (0,+\infty), \quad \text{ with } \quad  \lim_{\eta\to 0_+} \mu_K(\eta)=0,
     $$
     such that for all $(u,\xi),(\overline{u},\overline{\xi})\in K$ and  for almost every  $(x,t)\in Q,$ we have
     $$
     |a(x,t,u,\xi)-a(x,t,\overline{u},\overline{\xi})|\leq \mu_K(|u-\overline{u}|)+ C_K|\xi-\overline{\xi}|.
     $$
     \item[($C_3$)] $a(x,t,0,0)\in L^\infty(Q).$
      \end{itemize}
     Then, for every  $u\in W_p(Q),$ the composition 
     $$ 
     a(x,t,u(x,t),Du(x,t))\in VMO_x\cap L^\infty(Q).
     $$
 \end{lem}

 \begin{proof}
Let $M$ be  such that 
$$
\|u\|_{\infty,Q}+\|Du\|_{\infty,Q}\leq M.
$$
By  Sobolev embedding,  $u\in C^0(Q),$ and we denote    $\omega^x_u(r)$ the modulus of continuity of  $u$  with respect to $x$ 
\begin{equation} \label{eq-cont}
\omega^x_u(r) = \sup_{(y,\tau), (z,\tau) \in \I_r(x,t)} |u(y,\tau) - u(z,\tau)|.
\end{equation}

Without loss of generality, we  extend  $a$  as zero outside  $ \overline{Q}.$  We estimate the  mean  oscillation  in $x$ of the   composition $a(x,t,u(x,t),Du(x,t))$  over a parabolic cylinder: 
\begin{align*}
J :=& \, \Xint-_{\I_r(x,t)}
    \Xint-_{\B_r(x)} \big| a(y,\tau,u(y,\tau),Du(y,\tau)) 
     -a(z,\tau,u(z,\tau),Du(z,\tau)) \big| \, dz  dy  d\tau \\[6pt]
\leq  & \, \Xint-_{\I_r(x)}
    \Xint-_{\B_r(x)} \big| a(y,\tau,u(y,\tau),Du(y,\tau)) - a(z,\tau,u(y,\tau),Du(y,\tau)) \big|\, dzdyd\tau\\[6pt]
& +\Xint-_{\I_r(x)}
    \Xint-_{\B_r(x)}\big| a(z,\tau,u(y,\tau),Du(y,\tau)) - a(z,\tau,u(z,\tau),Du(z,\tau)) \big| \, dz  dy  d\tau\\
    =:& \, J_1+J_2.
\end{align*}
By  condition   $ (C_1),$  for any $R\geq r$ we have 
$$
J_1 \leq  a^{\#(x)}_{M}(R).
   $$
   For  $J_2,$  condition $ (C_2),$ the definition   \eqref{eq-cont}, and the regularity of $Du$  yield
$$
J_2 \leq \mu_K(\omega^x_u(R)) + C_K\, Du^{\#(x)}(R).
$$
Combining these estimates gives
$$
  {a}^{\#(x)}(R)
   \leq  a^{\#(x)}_{M}(R)+ \mu_K(\omega^x_u(R)) + C_K Du^{\#(x)}(R),
$$
which tends to $0$ as $R\to 0.$  
Thus $a(x,t,u(x,t),Du(x,t))\in VMO_x\cap L^\infty(Q)$ and the claim follows.
\end{proof}
The next  lemma is proved in 
\cite[Lemma 2]{PRS}.
\begin{lem}\label{lm2}
    Let $\U\subset \R\times\R^n$ be a bounded  open set, 
    and define  
    \begin{align*}
    \UU:=\Big\{ u\in W^{1,\infty}_x(Q) :& \quad \exists\,  K\subset \U \text{ compact such that } \\
    &(u,Du)\in K \text{ for a.a. }(x,t)\in Q\Big\}.
    \end{align*}
    Then:
    \begin{itemize}
        \item[(i)] $\UU$ is  open  in $W_x^{1,\infty}(Q)$.
        \item[(ii)] If $a\in\CC^1(Q\times\U),$  then the operator 
        $$
        \A(u):=a(x,t,u(x,t),Du(x,t))
        $$ 
        defines  a  $C^1$-map from $\UU$ to $L^\infty(Q)$.  Moreover, if $a\in \CC^{1,1}(Q\times \U),$ then the Frech\'et derivative 
        \begin{equation}\label{FDoper}
        \A'(u)v=D_u a(x,t,u,Du)\, v +  D_{\xi_l} a(x,t,u,Du)\, D_l v
        \end{equation}  
        is locally Lipschitz continuous.
     \end{itemize}
\end{lem}

\section{Local Solvability via the Implicit Function Theorem}

We define the following  superposition operators:
\begin{equation}\label{eq-lin}
\begin{split}
&\A^{ij}(u):=a^{ij}(x,t,u, Du),\qquad  \F(u):=f(x,t,u,Du)\\
 & \mathcal{P}(u):=D_tu-\A^{ij}(u) D_{ij}u- \F(u).
   \end{split}
\end{equation}
From  \eqref{eq-lin},  problem \eqref{PD} can be rewritten  in  abstract form:
\begin{equation}\label{PD1}
\mathcal{P}(u)= 0 \  \text{ in } Q, \qquad u=0 \  \text{ on } \partial_p Q.   
\end{equation}

By  formally linearizing  the operator $\mathcal{P}$  at a given function  $u=u_0$ through  its  Fréchet derivative,  we obtain the corresponding linearized  problem:
\begin{equation}\label{FLP}
\begin{cases}
   \mathcal{P}^\prime(u_0) v= D_t v-\A^{ij}(u_0)D_{ij}v&\\
   \phantom{\mathcal{P}^\prime(u_0) v} - 
    (\A^{ij\,\prime}(u_0)D_{ij}u_0+\F'(u_0))v=0 & \text{ in } Q\\
    v \in W_p(Q), & 
\end{cases}
\end{equation}
where  the derivatives  $\A^{ij\,\prime}$ and $\F'$ are understood   in the  Fréchet sense and  computed as in \eqref{FDoper}.

We now consider  the  problem \eqref{PD}  under   the following assumptions:
\begin{itemize}
    \item[$(H_1)$] The functions $a^{ij},  f:Q\times \R\times \R^n \to \R$ are $\CC^{1,1}$-Carath\'eodory functions.
     
      \item[$(H_2)$]  $u_0 \in  W_p(Q)$ is a solution of the problem \eqref{PD}.
     
    \item[$(H_3)$]  There exists a constant $\lambda>0$ such that the following uniform parabolicity  and symmetry conditions hold:
    \begin{align}\label{cnd3}
        \begin{cases}
        \lambda^{-1} |\eta|^2\leq a^{ij}(x,t,u_0, D u_0)\eta_i\eta_j\leq \lambda|\eta|^2, \quad &\forall\ \eta\in \R^n\\
        a^{ij}(x,t,u_0, D u_0)=a^{ji}(x,t,u_0, D u_0) &\forall \ i,j=1,\ldots,n. 
    \end{cases}
    \end{align}
     \item[$(H_4)$]      $a^{ij}(x,t,u_0, D u_0)\in VMO_x$  with  modulus
     $$
      a^{\#(x)}(R)= \sum_{i,j=1}^n  a^{ij\#(x)}(R),  
      $$
      and the source term satisfies $f(x,t,u_0,Du_0)\in L^p(Q).$

      \item[$(H_5)$] There are no non-trivial solutions $v\in W_p(Q)$ of  \eqref{FLP}.
\end{itemize}

Some remarks  regarding the above assumptions are in order.

First,  the symmetry of the matrix $\{a^{ij}(x,t,u_0,Du_0)\}_{i,j=1}^n$  ensures the  essential boundedness of its  components. Moreover, although the coefficients are defined only on the bounded cylinder $Q$, they can be extended as zero to the entire space  $\R^n\times \R$ without altering their $VMO_x$-modulus, as shown in  \cite{Jo} (see also \cite{Kr2}).

Under  assumptions $(H_1)$ and $(H_2)$  the operators $\mathcal{A}^{ij}(u)$ and $\mathcal{F}(u)$ defined by \eqref{eq-lin} are $C^1$-mappings with locally Lipschitz continuous Frèchet derivatives, as  established in  Lemma~\ref{lm2}. 
More precisely,
\allowdisplaybreaks
\begin{equation}\label{eq-operators}
\begin{split}
    &\mathcal{A}^{ij}(u), \, \mathcal{F}(u)  \in C^1( W_p(Q); L^\infty(Q)), \quad
    \mathcal{P}(u) \in C^1( W_p(Q); L^p(Q)), \\[4pt]
    &\|\A^{ij\,\prime}(u) - \A^{ij\,\prime} (\overline{u})\|_{\infty, Q} \leq L_{\mathcal{A}} \|u - \overline{u}\|_{W^{1,\infty}_x(Q)}, \\[4pt]
    &\|\mathcal{F}'(u) - \mathcal{F}'(\overline{u})\|_{\infty, Q} \leq L_{\mathcal{F}} \|u - \overline{u}\|_{W^{1,\infty}_x(Q)}, \\[4pt]
    &\|\mathcal{A}\| := \sum_{i,j=1}^n \|a^{ij}\|_{\CC^1}, \qquad 
    \|\mathcal{F}\| := \|f\|_{\CC^1}.
\end{split}
\end{equation}

The assumption $(H_2)$ is justified by the results in \cite{DK,KimKr,Re}. In fact,  under the assumptions $(H_3)$ and $(H_4),$  the unique solvability  in $W_p(Q),$  $p>1,$   
 of the linear Cauchy-Dirichlet    problem 
$$
\begin{cases}
\displaystyle u_t- a^{ij}(x,t)D_{ij}u=f(x,t) & (x,t)\in Q\\
u=0 & (x,t)\in \partial_p Q
\end{cases}
$$
follows from \cite[Theorem 6]{DK} together with the a priori estimate
$$
\|u\|_{W_p(Q)}\leq C\left(\|f\|_{p,Q}+\|u\|_{p,Q}\right)
$$
(see also 
  \cite[Lemma 2.2]{ANPS1}).  On the other hand, 
  for  $p>n+2,$   the maximum principle \cite[Theorem 1]{N}  implies
  $$
  \|u\|_{\infty,Q}\leq C\|f\|_{n+2,Q}.
  $$ 
  Therefore,  any solution
   $u\in W_p(Q)$ of  the linear problem satisfies the a priori estimate
\begin{equation}\label{rem1}
\|u\|_{W_p(Q)}\leq C\|f\|_{p,Q} \qquad p<n+2.
\end{equation}
Consequently, by the  Sobolev embedding theorems,  the solution $u_0$ is H\"older continuous together  with its gradient (see \cite{LSU}).
Hence, there exists an open set $\D\subset C([0,T]; C^1(\overline \Omega))$ such that $u_0\in \D\cap W_p(Q).$

This way, the 
assumptions $(H_3)$ and $(H_4)$ guarantee that the linear operator 
$$
D_t - \mathcal{A}^{ij}(u_0) D_{ij}: \ W_p(Q) \to L^p(Q), \quad   p>n+2,
$$
is an isomorphism.  By the  results of Dong-Kim  \cite[Theorem~6]{DK} and Krylov  \cite{Kr2},   linear parabolic operators with $VMO_x$ coefficients admit $L^p$-theory (see also \cite[Lemma~2.2]{ANPS1}) and the a priori estimate \eqref{rem1}
 ensures  bounded invertibility of the operator on $W_p(Q)$ and hence  validity of the    maximal regularity property.

The assumption $(H_5)$ ensures the   uniqueness  of the solution to   problem \eqref{FLP}. This condition is natural, since it guarantees uniqueness if the solution of the linearized problem. Indeed,
for linear parabolic equations whit $VMO_x$ coefficients we can rely on  the result of Dong and Kim \cite[Theorem~6]{DK} (see also Krilov \cite{Kr}), which asserts existence and uniqueness of strong solution $u\in W^{2,1}_p(Q), p>1,$ of the Dirichet problem 
\begin{equation*}
\begin{cases}
  D_t u-a^{ij}(x,t)D_{ij}u +\lambda u=f(x,t)\quad & (x,t)\in Q,\\ 
  u=0 & (x,t)\in \partial Q
\end{cases}
\end{equation*}
for which the a priori estimate holds: 
$$
\|D_t u\|_{L^p(Q)}+\|D^2u\|_{L^p(Q)}\leq  c \|f\|_{L^p(Q)}.
$$

Furthermore, this result was used  in  \cite{Re}) to obtain the  unique solvability of the  following Dirichlet problem for quasilinear equations
\begin{equation*}
    \begin{cases}
   D_tu-a^{ij}(x,t,u)D_{ij}u=f(x,t,u,Du)  \qquad &\text{a.e in }Q\,,\\
   u(x,t)=0 & \text{on } \partial_pQ\,.
    \end{cases}
\end{equation*}
We  observe that  conditions $(H_1),(H_3),$ and $(H_4)$ ensure  the validity of the regularity assumptions  required in \cite{Re}.

Finally, 
in the  case  where 
$$
\A^{ij\,\prime}(u_0)D_{ij}u_0+\F'(u_0)\leq 0,
$$ 
the maximum principle \cite{N} guarantees  that $(H_5)$ is automatically  satisfied.

Without loss of generality, we take $u_0=0\in \D$ and fix it in the coefficients of \eqref{PD1} 
  obtaining  the following auxiliary    linear  problem:
\begin{equation}\label{PD1_aux}
    D_t w - \mathcal{A}^{ij}(0) D_{ij} w = \mathcal{F}(0) \  \text{ in } Q,\qquad w=0 \ \text{ on } \partial_p Q
\end{equation}
which admits  unique solution  due to the assumptions $(H_2)-(H_5)$  as explained above. 

Let $U_0\subset \D$ be a neighbourhood of  $u_0=0$ and  $ V_0\subset \D$ a neighbourhood of zero,  such that  
$$
\{u + w : (u, w) \in U_0 \times V_0\} \subset \D.
$$
We interpret  $u$ as a perturbation of the solution   $u_0=0$, and $w$ as a small correction.  Then the pair $(u, w)$ belongs to $(U_0 \cap W_p(Q)) \times V_0$, and we consider the perturbed problem
\begin{equation}\label{NLP}
    D_t(u + w) - \mathcal{A}^{ij}(u + w) D_{ij}(u + w) = \mathcal{F}(u + w).
\end{equation}

We now introduce  the operators:
\allowdisplaybreaks
\begin{equation}\label{oper}
\begin{split}
    \overline{\mathcal{A}}_{ij}(u, w) := &\mathcal{A}_{ij}(u + w) = a^{ij}(x, t, u + w, D(u + w)), \\[4pt]
    \overline{\mathcal{F}}(u, w) :=& \mathcal{F}(u + w) - \mathcal{F}(0) + (\mathcal{A}_{ij}(u + w) - \mathcal{A}_{ij}(0)) D_{ij} w \\[4pt]
    = &f(x, t, u + w, D(u + w)) - f(x, t, 0, 0) \\
    +& \left(a^{ij}(x, t, u + w, D(u + w)) - a^{ij}(x, t, 0, 0)\right) D_{ij} w.
\end{split}
\end{equation}

By assumption   $(H_1)$ and Lemma~\ref{lm2}, these are $C^1$-maps:
\begin{align*}
    \overline{\mathcal{A}}_{ij}(u, w) &\in C^1((U_0 \cap W_p(Q)) \times V_0; L^\infty(Q)), \\[4pt]
    \overline{\mathcal{F}}(u, w) &\in C^1((U_0 \cap W_p(Q)) \times V_0; L^\infty(Q)).
\end{align*}

Using \eqref{PD1_aux}, we can rewrite the problem \eqref{NLP} in the equivalent perturbed form:
\begin{equation}\label{PD_pert}
    D_t u - \overline{\mathcal{A}}_{ij}(u, w) D_{ij} u = \overline{\mathcal{F}}(u, w).
\end{equation}

Since $\overline{\mathcal{A}}_{ij}(0, 0) = \mathcal{A}^{ij}(0)$ and $\overline{\mathcal{F}}(0, 0) = 0$, it follows that the pair $(u, w) = (0, 0) \in U_0 \times V_0$ is a solution of \eqref{PD_pert}.

The following result establishes the   smooth dependence of the solution to \eqref{PD1} under small perturbation of $u_0$.

\begin{thm}\label{Thm4.4}
Let the hypotheses $(H_1)-(H_5)$ hold and 
    let $U_0$ and $V_0$ be as defined above. Then there exist neighbourhoods $U_{1,M} \subset U_0$ and $V_{1,\varepsilon} \subset V_0$, with $(0,0) \in U_{1,M} \times V_{1,\varepsilon}$, and a solution map $\Phi: V_{1,\varepsilon} \to W_p(Q)$ of class $C^1$ such that the pair $(u,w) \in U_{1,M} \times V_{1,\varepsilon}$ is a solution of \eqref{PD_pert} if and only if $u = \Phi(w)$.
\end{thm}

\begin{proof}
By $(H_3)$ and $(H_4)$ we have that  $\overline{\A}_{ij}(0,0)=a^{ij}(x,t,0,0) \in VMO_x\cap  L^\infty(Q).$ 
We are going to evaluate the mean oscillation of the superposition operators 
$$
\overline{\A}_{ij}(u,w)= a^{ij}(x,t,u(x,t)+w(x,t), Du(x,t)+Dw(x,t))
$$
 for $(u,v)$ sufficiently close to $(0,0).$   Without loss of generality we extend $a^{ij}$ as zero outside $\overline Q.$
According to 
Lemma~\ref{lm1}  and $(H_1)$ we have  $\overline{\A}_{ij}(u,0)=\A^{ij}(u)\in VMO_x \cap L^\infty(Q)$ for every $u\in W_p(Q)$.
Define the sets:
\begin{equation}\label{U1-W1}
\begin{split}
    U_{1,M} &= \left\{ u \in U_0 : \|u\|_{C([0,T], C^1(\overline{\Omega}))} < M \right\}, \\
    V_{1,\varepsilon} &= \left\{ w \in V_0 : \|w\|_{C([0,T], C^1(\overline{\Omega}))} < \varepsilon \right\}.
\end{split}
\end{equation}

Using $(H_1)$ and the continuity of $u$ and $w,$ we estimate
\allowdisplaybreaks
\begin{align*}
   \mathcal{J}&:= \dashint_{\I_r(x,t)} \dashint_{\B_r(x)} \big| a^{ij}(y,\tau,u(y,\tau) + w(y,\tau), Du(y,\tau) + Dw(y,\tau)) \\
    &\qquad -  a^{ij}(z,\tau,u(z,\tau)+w(z,\tau), Du(z,\tau)+Dw(z,\tau)) \big| \,dz dy d\tau\\
    & \leq \dashint_{\I_r(x,t)} \dashint_{\B_r(x)} \big|
     a^{ij}(y,\tau,u(y,\tau) + w(y,\tau), Du(y,\tau) + Dw(y,\tau))\\
      &\qquad -
      a^{ij}(y,\tau,u(z,\tau) + w(z,\tau), Du(z,\tau) + Dw(z,\tau))
    \big|\,dzdyd\tau\\
     &\qquad+ \dashint_{\I_r(x,t)} \dashint_{\B_r(x)} \big|    a^{ij}(y,\tau,u(z,\tau) + w(z,\tau), Du(z,\tau) + Dw(z,\tau))\\
    &\qquad -  a^{ij}(z,\tau,u(z,\tau)+w(z,\tau), Du(z,\tau)+Dw(z,\tau)) \big| \,dz dy d\tau\\
     & \leq L \Bigg(\dashint_{\I_r(x,t)} \dashint_{\B_r} |u(y,\tau)-u(z,\tau)| +|w(y,\tau)-w(z,\tau)|\, dzdyd\tau\\
   &\qquad + \dashint_{\I_r(x,t)} \dashint_{\B_r} |Du(y,\tau)-Du(z,\tau)|\, dzdyd\tau\\
     &\qquad+ \dashint_{\I_r(x,t)} \dashint_{\B_r} |Dw(y,\tau)-Dw(z,\tau)|\, dzdyd\tau\Bigg)\\
     &\qquad + \dashint_{\I_r(x,t)} \dashint_{\B_r} \big|    a^{ij}(y,\tau,u(z,\tau) + w(z,\tau), Du(z,\tau) + Dw(z,\tau))\\
    &\qquad -  a^{ij}(z,\tau,u(z,\tau)+w(z,\tau), Du(z,\tau)+Dw(z,\tau)) \big| \,dz dy d\tau.
\end{align*}
Then 
\begin{align*}
 \mathcal J    
    &\leq    L\Big(\omega^x_u(r)+ \omega_w^x(r)  + \osc_x (Du, \I_r(x,t)) + \osc_x (Dw, \I_r(x,t)) \Big)\\
    & +\dashint_{\I_r(x,t)} \dashint_{\B_r} \big|    a^{ij}(y,\tau,u(z,\tau) + w(z,\tau), Du(z,\tau) + Dw(z,\tau))\\
    &\qquad -  a^{ij}(z,\tau,u(z,\tau)+w(z,\tau), Du(z,\tau)+Dw(z,\tau)) \big| \,dz dy d\tau.
\end{align*}
Here  $L$ is  the Lipschitz constant of the coefficients  
$\{a^{ij}\}$ which can be  taken as  $\max_{ij} L_{a^{ij}}$ (cf. Definition~\ref{del-C11}). Applying  Lemma~\ref{lm1} to the last term  we get 
$$
   \overline{\A}_{ij}^{\#(x)}(R) := 
   \sup_{r\leq R}\sup_{(x,t)\in Q} \mathcal{J} \to 0 \quad \text{ as } \ R\to 0
    $$
    that gives $ \overline{\A}_{ij}(u,w)\in VMO_x\cap L^\infty(Q).$

Thus  solutions $(u,w)$ of \eqref{PD_pert} can be sought in a neighbourhood of $(0,0)$ in the space $(U_{1,M} \cap W_p(Q)) \times (V_{1,\varepsilon}\cap W_p(Q)),$ and we may  apply the Implicit Function Theorem 
\cite[Theorem~4B]{Z}.

Since $W_p(Q)$ embeds continuously into $C([0,T], C^1(\overline{\Omega})),$ the set $U_{1,M} \cap W_p(Q)$ is open in $W_p(Q)$. Define 
$$
\mathcal{P}(u,w) := D_t u - \overline{\A}_{ij}(u,w) D_{ij} u - \overline{\F}(u,w),
$$
which is a $C^1$-map from $(U_{1,M} \cap W_p(Q)) \times (V_{1,\varepsilon}\cap W_p(Q))$ into $L^p(Q)$.

The  Fréchet derivative of $\mathcal{P}$ with respect to $u$ at $(0,0)$ is 
$$
\mathcal{P}'(0,0)(v) = D_t v - \overline{\A}_{ij}(0,0) D_{ij} v  \qquad v \in W_p(Q)
$$
since $\overline{\F}(0,0)=0.$
By $(H_3)$ and Lemma~\ref{lm2}, the operator  $\mathcal{P}'(0,0)$ is a linear isomorphism. Therefore, the Implicit Function Theorem  implies that there exist $M, \varepsilon>0$ and  a unique $C^1$-map  $\Phi: V_{1,\varepsilon}\to U_{1,M},$ such that  $u = \Phi(w)$  for all  $w \in V_{1,\varepsilon}\cap W_p(Q)$ solving  \eqref{PD_pert}.
\end{proof}

The solution to the  problem \eqref{PD1} cannot, in general, be expected to exist for an arbitrarily long time interval unless additional conditions on the data are imposed.

Let $\tau \in (0,T)$, define $Q_\tau = \Omega \times (0,\tau),$  denote by $u_\tau$ the restriction of $u$ on the  cylinder $Q_\tau,$ and set  $\D_\tau=\{u_\tau : u\in \D\}.$  
We  add  the following  assumption.
\begin{itemize}
\item[$(H_6)$] The mappings  $\A^{ij}(u)$ and $\F(u)$ are {\it Volterra operators,} that is,    for any $\tau\in (0,T)$ and  any $u,v\in \D\cap W_p(Q)$ it holds
$$
\text{ if } \ u_\tau=v_\tau \ \text{ then } \ 
\A^{ij}(u_\tau)= \A^{ij}(v_\tau) \ \text{ and } \  \F(u_\tau)= \F(v_\tau).
$$

\end{itemize}

Note that for all $u\in \D\cap W_p(Q) ,$ the linear continuous operators $\A^{ij\prime}(u)$ and $\F^\prime(u)$ also verify the Volterra property, since the Fréchet derivative of Volterra mapping also has the Volterra property (see for instance \cite{GRk1,Z}).

The next theorem claims local in time solvability of \eqref{PD1}.

\begin{thm}\label{thm_tau}
    Suppose that conditions $(H_1)-(H_6)$ are satisfied. Then there exists $\tau\in (0,T]$, depending on the data, such that 
there exists at least one solution
$ 
u_\tau \in \D_\tau \cap W_p(Q_\tau)
$ 
to the problem \eqref{PD1} restricted on $Q_\tau$.
\end{thm}

\begin{proof}
    Let $v \in W_p(Q)$ be a solution of the auxiliary problem \eqref{PD1_aux}. Since $v(x,0)=0$ and $v$ is continuous in time, it remains small for sufficiently small $\tau > 0.$  More precisely,  by the completely continuous embedding 
(see \cite[Lemma 3.3]{LSU}):
    \begin{equation}\label{eq-emb}
    W_p(Q_\tau) \hookrightarrow C^{0,\alpha}([0,\tau]; C^1(\overline{\Omega})), \quad 0 \leq \alpha \leq 1 - \frac{n+2}{p}, \quad p > n+2,
    \end{equation}
    we obtain, for each $0 < s \leq \tau,$
    \begin{align*}
        \|v(\cdot,s) - v(\cdot,0)\|_{C^1(\overline{\Omega})}
        &\leq \tau^{\alpha/2} \|v\|_{C^{0,\alpha}([0,\tau]; C^1(\overline{\Omega}))} \leq C \tau^{\alpha/2} \|v\|_{W_p(Q_\tau)},
    \end{align*}
    with $C$  independent of $\tau$. Hence 
    \begin{equation}\label{eq-small}
    \|v\|_{C([0,\tau]; C^1(\overline{\Omega}))}\leq  C \tau^{\alpha/2} \|v\|_{W_p(Q_\tau)}<\varepsilon
    \end{equation}
    provided  $\tau\in (0,T]$ is sufficiently small.

Next, let 
$\theta \in C^\infty(\mathbb{R})$ be a cut-off function, $0 \leq \theta \leq 1$, such that for suitable $0<\delta<\tau<T$ we have 
    $$
    \theta(s) =
    \begin{cases}
        1 & \text{for } 0 \leq s \leq \delta , \\
        0 & \text{for } s \geq \tau.
    \end{cases}
    $$
    Then $\theta(s) v \in C([0,T], C^1(\overline{\Omega}))$ and $\theta v\in V_{1,\varepsilon}\cap W_p(Q)$   by \eqref{eq-small}.

    Choosing $w = \theta v$ in \eqref{PD_pert} and $u\in U_{1,M} \subset \D$ so  that $u + \theta v \in \D_\tau$ for all $u \in U_{1,M}\cap W_p(Q_\tau)$, we conclude from Theorem~\ref{Thm4.4} that the function $u = \Phi(\theta v)$ solves 
    $$
        D_t u - \overline{\A}_{ij}(u,\theta v) D_{ij} u = \overline{F}(u,\theta v) \quad   \text{ in  } Q_\tau.
 $$
    By equations \eqref{PD1_aux} and \eqref{oper}, this problem can be rewritten as  
$$
        D_t (u + \theta v) - \A^{ij}(u + \theta v) D_{ij}(u + \theta v) = \F(u + \theta v) \quad \text{ in } Q_\tau.
$$
    Define
    $$
    u_\tau = (u + \theta v)\big|_{Q_\tau}.
    $$
then $u_\tau$ is a solution of the restricted problem 
    \begin{align}\label{3.10}
        \begin{cases}
            D_t u_\tau - \A^{ij}(u_\tau) D_{ij} u_\tau = \F(u_\tau) & \text{in } Q_\tau, \\
            u_\tau \in \D_\tau \cap  W_p(Q_\tau)
        \end{cases}
    \end{align}
    that establishes  existence of a local in time solution. 
\end{proof}

\begin{thm}
    Suppose that $(H_1)-(H_6)$ hold. If $u,v \in \D\cap  W_p(Q)$ are two solutions of \eqref{PD1}, then $u \equiv v$.
\end{thm}

\begin{proof}
Since $u(0)=v(0)=0$ and 
    $u,$ $v$ are continuous in time,  define 
    $$
        t^* = \sup\{\, t \in [0,T] : u(s) = v(s), \ \forall\, 0 \leq s \leq t \,\}.
    $$
    Clearly $t^*\geq 0.$
     We aim to show that $t^* = T$. 
     
     Assume the contrary, that $t^* < T$. Choose $\tau \in (t^*, T)$ and
     set
    $$
    Q_\tau^*=\Omega\times (t^*,\tau), \qquad 
        \eta(x,t) := u(x,t) - v(x,t). 
    $$
   Then  $\eta(x,t^*) = 0$ and $\eta\not\equiv 0$ on $Q^*_\tau$ for small $\tau-  t^*$.

   Thus, $\eta$ solves \eqref{PD1} in $Q_\tau^*$ with zero initial data.

By the parabolic embedding (cf.~\eqref{eq-emb} applied on $(t^*,\tau)$), we have for $s\in(t^*,\tau]$,
$$
\|\eta(\cdot,s)\|_{C^1(\overline\Omega)} \le C (\tau-t^*)^{\alpha/2}\|\eta\|_{W_p(Q^*_\tau)}.
$$
Hence 
\begin{equation}\label{eq-epsilon}
\|\eta\|_{C([t^*,\tau];C^1(\overline\Omega))}\leq C (\tau-t^*)^{\alpha/2}\|\eta\|_{W_p(Q^*_\tau)}<\varepsilon
\end{equation} 
for $\tau$  close enough to $t^*.$

Subtracting the equations for $u_\tau$ and $v_\tau$ (the restrictions of $u,v$ on $Q_\tau^*$) gives
$$
D_t \eta_\tau - \A^{ij}(u_\tau) D_{ij}\eta_\tau
= \F(u_\tau)-\F(v_\tau) + \big(\A^{ij}(v_\tau)-\A^{ij}(u_\tau)\big)D_{ij}v_\tau .
$$
    Using the mean value theorem (cf. \cite{GRk1}), this can be rewritten   as 
    \begin{align}
    \nonumber
        D_t \eta_\tau - \A^{ij}(u_\tau) D_{ij} \eta_\tau
        &= \int_0^1 \F'\left(v_\tau + \sigma \eta_\tau\right)\,d\sigma \cdot \eta_\tau \\
        \label{eq-N}
        &+ D_{ij} v_\tau \int_0^1 \A^{ij\,\prime}\left(v_\tau + \sigma \eta_\tau\right)\,d\sigma \cdot \eta_\tau  =:   \mathcal{N} \eta_\tau.
    \end{align}

    By \eqref{eq-emb},  the operator  $\mathcal{N}$ is a linear completely continuous Volterra operator 
    $$
    \mathcal{N}: C((t^*,\tau];C^1(\overline\Omega)) \to L^p(Q_\tau^*).
    $$
    Since  $\mathcal{L}_0:=D_t-\mathcal{A}^{ij}(u_\tau)$ is an isomorphism,  then 
    $$
    \mathcal{L}:=D_t-\mathcal{A}^{ij}(u_\tau) -\mathcal N=\mathcal{L}_0-\mathcal{N} : \ W_p(Q_\tau^*)\to L^p(Q^*_\tau)
    $$
    is Fredholm of index zero.
    It remains to show that $\ker\mathcal L=\{0\}$. Let $\mathcal L_0^{-1}$ denote the bounded inverse operator  of the principal part
$
    \mathcal L_0$ 
        whose boundedness follows from the maximal regularity (isomorphism property)  established above. From the representation
    $$
    \mathcal L = \mathcal L_0(I - \mathcal L_0^{-1}\mathcal N),
    $$
    any $\eta\in\ker\mathcal L$ satisfies
    $ 
    \eta = \mathcal L_0^{-1}\mathcal N\eta .
    $ 
    
    To estimate the operator norm of $\mathcal L_0^{-1}\mathcal N$ we use  its  mean-value representation \eqref{eq-N} and the local Lipschitz properties of $\F'$ and $\A^{ij\,\prime}$ (see Lemma~\ref{lm2}). Hence there is a constant $C$ (depending on the $W_p(Q^*_\tau)$ norm of   $v_\tau$ ) such that for every $w\in W_p(Q_\tau^*),$ by \eqref{eq-epsilon} we get
    $$
    \|\mathcal N w\|_{L^p(Q_\tau^*)} \le C\|w\|_{C([t^*,\tau];C^1(\overline\Omega))}\leq 
     C_1 (\tau-t^*)^{\alpha/2} \|w\|_{W_p(Q_\tau^*)}
  $$
  with $C_1$ independent of $\tau.$
    Therefore
    $$
    \|\mathcal L_0^{-1}\mathcal N\|_{\mathcal L(W_p,L^p)} \le \|\mathcal L_0^{-1}\|_{\mathcal L(L^p,W_p)}\, C_1  (\tau-t^*)^{\alpha/2} =: C_2 (\tau-t^*)^{\alpha/2}.
    $$
    Since $C_2$ is independent of $\tau$, we can choose $\tau-t^*$ small enough such  that
    $$
    \|\mathcal L_0^{-1}\mathcal N\| < 1.
 $$
    For such $\tau$ the operator $I - \mathcal L_0^{-1}\mathcal N$ is invertible on $W_p(Q_\tau^*)$  and hence
    $$
    \mathcal L = \mathcal L_0 (I - \mathcal L_0^{-1}\mathcal N)
    $$
    is invertible. In particular $\ker\mathcal L=\{0\}$ in $Q_\tau^*$, so the only solution of the homogeneous problem in $W_p(Q_\tau^*)$ with zero initial data is $\eta\equiv 0$.
    
    This contradicts the definition of $t^*$ (recall we supposed $\eta\not\equiv0$ on $Q_\tau^*$), therefore the assumption $t^*<T$ is false and thus $t^*=T$. Consequently $u\equiv v$ in $Q$, which proves uniqueness.
\end{proof}

We are now in a position to establish our main result. The following theorem is an application of the Implicit Function Theorem,  and it asserts that, for all sufficiently small perturbations $\tilde{a}_{ij}$ and $\tilde{f}$ of the coefficients $a^{ij}$ and $f$, the perturbed problem \eqref{PD_pert} admits a unique solution. Moreover, this solution depends in a $C^1$-smooth manner on the perturbations.

\begin{thm}\label{thm_3.7}
Let $\mathcal{U} \subset \R \times \R^n$ be an open set, and let $\tau \in (0,T)$ be as in Theorem~\ref{thm_tau}. Suppose that $u_{0\tau}$ is a local-in-time solution to \eqref{PD1} under the assumptions $(H_1)-(H_6)$.  Let $K \subset \mathcal{U}$ be compact such that $(u_{0\tau}, Du_{0\tau}) \in K$ for almost every $(x,t) \in Q_\tau$. Then there exist:
\begin{itemize}
    \item a neighbourhood $\mathcal V_\tau \subset \CC^1(Q_\tau \times \overline{\mathcal{U}})^{n^2} \times \CC(Q_\tau \times \overline{\mathcal{U}})$ of $(0,0)$,
    \item a neighbourhood $\mathcal W_\tau \subset \D_\tau \cap W_p(Q_\tau)$ of $u_{0\tau}$,
    \item a $C^1$-map $\Phi: \mathcal V_\tau \to \mathcal W_\tau$ with $\Phi(0,0) = u_{0\tau}$,
\end{itemize}
such that for each  $(\{\tilde{a}^{ij}\}_{i,j=1}^n, \tilde{f}) \in \mathcal V_\tau$ a function $u_\tau \in \mathcal W_\tau$  is a solution to the perturbed problem
\begin{equation}\label{PD_pert_tau}
\begin{cases}
D_t u_\tau - \left(a^{ij}(x,t,u_\tau,Du_\tau) + \tilde{a}^{ij}(x,t,u_\tau,Du_\tau)\right) D_{ij} u_\tau  \\
\hfill =f(x,t,u_\tau,Du_\tau) + \tilde{f}(x,t,u_\tau,Du_\tau), & \text{in } Q_\tau, \\
u_\tau = 0, & \text{on } \partial_p Q_\tau,
\end{cases}
\end{equation}
if and only if $u_\tau = \Phi(\{\tilde{a}^{ij}\}_{i,j=1}^n, \tilde{f})$.
\end{thm}

\begin{proof}
Let $\widetilde{\mathbf{A}} = \{\tilde{a}^{ij}\}_{i,j=1}^n \in 
\CC^1(Q_\tau \times \overline{\mathcal{U}})^{n^2}$ denote the perturbation matrix. Define the set
$$
\mathcal{U}_\tau := \left\{\, u_\tau \in D_\tau \cap W_p(Q_\tau) : \  (u_\tau, Du_\tau) \in K \subset \mathcal{U} \, \right\}.
$$
The set $\mathcal{U}_\tau$ is open in $W_p(Q_\tau)$. By assumption $(H_1)$ and Lemma~\ref{lm2}, there exist $C^1$-maps
\begin{align*}
\A^{ij} & : \CC^1(Q_\tau \times \overline{\mathcal{U}})^{n^2} \times \mathcal{U}_\tau \to L^\infty(Q_\tau), \\
\F & : \CC^1(Q_\tau \times \overline{\mathcal{U}}) \times \mathcal{U}_\tau \to L^\infty(Q_\tau),
\end{align*}
such that
\begin{align*}
\A^{ij}(\tilde{a}^{ij}; u_\tau) & := a^{ij}(x,t,u_\tau,Du_\tau) + \tilde{a}^{ij}(x,t,u_\tau,Du_\tau), \quad \A^{ij}(0;u_\tau)=\A^{ij}(u_\tau)\\
\F(\tilde{f}; u_\tau) & := f(x,t,u_\tau,Du_\tau) + \tilde{f}(x,t,u_\tau,Du_\tau) \quad \F(0;u_\tau)=\F(u_\tau).
\end{align*}
Hence, the operator in \eqref{PD_pert_tau} is equivalent to
\begin{equation}\label{PD_pert1}
\begin{cases}
\widetilde{\mathcal{P}}(\tilde{a}^{ij}, \tilde{f}; u_\tau) := D_t u_\tau - \A^{ij}(\tilde{a}^{ij}; u_\tau) D_{ij} u_\tau - \F( \tilde{f}; u_\tau) = 0, \\
\widetilde{\mathcal{P}} \in C^1\big( \CC^1(Q_\tau \times \overline{\mathcal{U}})^{n^2} \times \CC^1(Q_\tau \times \overline{\mathcal{U}}) \times \mathcal{U}_\tau; L^p(Q_\tau) \big).
\end{cases}
\end{equation}
Note that
$$
\widetilde{\mathcal{P}}(0,0; u_{0\tau}) 
= D_t u_{0\tau} - \A^{ij}( u_{0\tau}) D_{ij} u_{0\tau} - \F( u_{0\tau})  = 0.
$$

We aim to solve \eqref{PD_pert1} for $u_\tau$ near $u_{0,\tau} $ with respect to the norm $W_p$ using  the Implicit Function Theorem. For this, we compute the Fréchet derivative of  $\widetilde{\mathcal{P}}$ with respect to $u,$ evaluated at $(0,0;u_{0\tau})$:
\begin{align*}
\widetilde{\mathcal{P}}^\prime(0,0; u_{0\tau}) v_\tau 
&= \big( D_t  - \A^{ij}( u_{0\tau}) D_{ij} \big)v_\tau \\
& - \big(\F^\prime(u_{0\tau})  + \A^{ij\,\prime}(u_{0\tau}) D_{ij} u_{0\tau} \big) v_\tau.
\end{align*}
This operator  is the sum of the following two mappings:
\begin{align}
v_\tau\mapsto \big( D_t - \A^{ij}( u_{0\tau}) D_{ij}\big) v_\tau: & \quad    \D_\tau \cap W_p(Q_\tau) \mapsto  L^p(Q_\tau), \label{map1} \\
v_\tau\mapsto \big(\F^\prime ( u_{0\tau})  + \A^{ij\prime} ( u_{0\tau}) D_{ij} u_{0\tau}\big)  v_\tau: & \quad   \D_\tau \cap W_p(Q_\tau) \mapsto L^p(Q_\tau). \label{map2}
\end{align}

As discussed above, the operator \eqref{map1} is isomorphism thanks to $(H_3),$ $(H_4),$ $(H_6),$ and the maximum principle.
   The second  mapping   \eqref{map2} is bounded  multiplication operator, which  is compact due to the compactness of the  embedding $W_p (Q_\tau)\hookrightarrow L^p(Q_\tau).$ 
   Hence   
 $\widetilde{\mathcal{P}}^\prime(0,0; u_{0\tau})$ is a Fredholm operator  of index zero from $\D_\tau \cap W_p(Q_\tau)$ into $L^p(Q_\tau),$ being a  sum of an isomorphism and a compact operator.

Moreover, hypothesis $(H_5)$ ensures that this operator is injective.  Indeed, by $(H_5)$ the unperturbed linearized problem
\begin{equation}
\label{q-unprt}
\begin{cases}
D_t v_\tau - \A^{ij}( u_{0\tau}) D_{ij} v_\tau&\\ 
\qquad\qquad\qquad - \big(\F^\prime( u_{0\tau}) + \A^{ij\prime}(u_{0\tau}) D_{ij} u_{0\tau}\big)v_\tau = 0, &  \text{ in } Q_\tau\\
 v_\tau=0 & \text{ on } \partial_p Q_\tau,
\end{cases}
\end{equation}
admits only the trivial solution $v_\tau \equiv 0$ in $\D_\tau \cap W_p(Q_\tau)$.  
In other words,
$$
\ker \big(\widetilde{\mathcal{P}}^\prime(0,0; u_{0\tau})\big) = \{0\}.
$$
Since $  \widetilde{\mathcal{P}}^\prime(0,0; u_{0\tau})$ is Fredholm of index zero and injective, it is also surjective.  Thus 
$$
\widetilde{\mathcal{P}}^\prime(0,0; u_{0\tau}) \in \mathbf{Iso}(\D_\tau \cap W_p(Q_\tau), L^p(Q_\tau)).
$$

Applying the Implicit Function Theorem to  $\widetilde{\mathcal{P}},$ we conclude that   there exist neighbourhoods $\mathcal V_\tau$ and $\mathcal W_\tau$ of $(0,0)$ and $u_{0\tau},$ respectively, and a $C^1$ map $\Phi:\mathcal V_\tau\to \mathcal W_\tau ,$ such that for all $(\{\tilde{a}^{ij}\}, \tilde{f}) \in \mathcal V_\tau $  the function $u_\tau$ is a solution of  \eqref{PD_pert1} if and only if  
$$
u_\tau =\Phi(\{\tilde{a}^{ij}\}_{i,j=1}^n, \tilde{f}).
$$
\end{proof}


\section{Application of the Newton Iteration Procedure}

The previous results allow us to apply the {\it Newton Iteration Procedure} to construct a sequence of approximate solutions converging to the actual solution of the quasilinear parabolic Cauchy-Dirichlet problem \eqref{PD}.
 For,  consider its formal linearization at  $u=u_0$:

\begin{equation} \label{PD-linearized}
\begin{cases}
\mathcal P^\prime(u_0)v=  D_t v - a^{ij}(x, t, u_0, Du_0) D_{ij} v\\
\quad - \left[ D_{\xi_l}a^{ij}(x,t,u_0,Du_0)D_{ij}u_0 + D_{\xi_l}f(x,t,u_0,Du_0)\right]D_l v\\
\quad - \left[ D_u a^{ij}(x,t,u_0,D u_0)D_{ij}u_0 + D_u f(x,t,u_0,Du_0) \right] v = 0 & \text{a.e. in } Q, \\
v(x, t) = 0 & \text{on } \partial_p Q.
\end{cases}
\end{equation}

We define the Newton iteration scheme in the following way: for a given initial function $u_1$, define the sequence $\{u_k\}$ by solving the linear problems
\begin{equation}\label{PD_k}
\begin{cases}
\mathcal{P}'(u_k)(u_{k+1} - u_k) = -\mathcal{P}(u_k) & \text{ in } Q, \\
u_{k+1} = 0 & \text{ on } \partial_p Q.
\end{cases}
\end{equation}

Denote by $ \mathscr{D}_p $ the set  of all symmetric matrix  $\{a^{ij}\}$ such that
$$
\mathscr{D}_p := \Big\{ \{a^{ij}\}_{i,j=1}^n \in L^\infty(Q)^{n^2} : \exists\, \lambda > 0  \text{ s.t. } a^{ij}(x,t)\eta_i\eta_j \geq \lambda|\eta|^2 \text{ a.e. in } Q,\, \forall \eta \in \mathbb{R}^n \Big\},
$$
and for which the operator 
$$
u \mapsto (D_t - a^{ij} D_{ij})u : W_p(Q) \to L^p(Q)
$$ 
is an isomorphism.

In the following theorem we follow  \cite[Theorem~15.6]{deimling} and  \cite[Theorem~5A]{Z}) to investigate the   convergence of the Newton iteration \eqref{PD_k}.
\begin{thm}
Suppose the following assumptions  hold
\begin{itemize}
    \item[$(H'_1)$] The functions $a^{ij},  f:Q\times \R\times \R^n \to \R$ are $\CC^{1,1}$-Carath\'eodory functions.
     
      \item[$(H'_2)$]  $u_0 \in  W_p(Q)$ is a solution of the problem \eqref{PD}.
     
     \item[$(H'_3)$]      $\{a^{ij}(x,t,u_0, D u_0)\}_{i,j=1}^n\in \mathscr{D}_p$

      \item[$(H'_4)$] There are no non-trivial solutions $v\in W_p(Q)$ of  \eqref{PD-linearized}.
\end{itemize}

 Then there exists a neighbourhood $\mathcal W \subset W_p(Q)$ of $u_0$ such that for every $u_1 \in \mathcal W$, the Newton sequence
\begin{equation}\label{eq-NI}
u_{k+1}    =u_k -\mathcal{P}'(u_k)^{-1} \mathcal{P}(u_k), \qquad k\geq 1 
\end{equation}
generated by the Newton iteration \eqref{PD_k} is well-defined, $u_k\in \mathcal W$ for all $k\in \mathbb{N},$ and converges to $u_0$ in $W_p(Q)$ with quadratic rate.
\end{thm}

\begin{proof}
Assumptions $(H'_3)$ and $(H'_4)$ ensure that for each $u \in W_p(Q)$, the operator $\mathcal{P}'(u)$ belongs to $\mathrm{Iso}(W_p(Q), L^p(Q))$ and that the map $u \mapsto \mathcal{P}'(u)$ is Lipschitz continuous in a neighbourhood $\mathcal W$ of $u_0$.
In fact, by  the compactness of the  embedding $W_p(Q) \hookrightarrow W^{1,\infty}_x(Q)$ (valid for $p>n+2$) and by \eqref{eq-operators}, we can estimate for any  $u,v, w \in  W_p(Q)$ 
\begin{equation}\label{eq-Pprimo}
\begin{cases}
\| (\mathcal{P}'(u) - \mathcal{P}'(v)) w \|_{p,Q}
&\leq C \|u - v\|_{W^{1,\infty}_x(Q)} \|w\|_{W_p(Q)} \\
&\leq L \|u - v\|_{W_p(Q)} \|w\|_{W_p(Q)},
\end{cases}
\end{equation}
where the  constant $L$ depends on $L_{\mathcal A}$ and  $L_{\mathcal F}$ from \eqref{eq-operators}.

Since  $\mathcal{P}'(u_0)$ is an isomorphism from $W_p(Q)$ onto $L^p(Q)$ we  set
$$
M:=\|\mathcal{P}'(u_0)^{-1}\|_{\mathcal{L}(L^p(Q),W_p(Q))} <\infty.
$$
Let 
$$
\mathcal W_{r_0}:= \{\, u\in W_p(Q) : \|u-u_0\|_{W_p(Q)}< r_0\,\}.
$$ 
By the continuity of $\mathcal{P}'$ we may choose   $r_0>0$ small enough   in order to ensure  that for any $u\in \mathcal W_{r_0}$
the operator $\mathcal{P}'(u)$ is invertible and
\begin{equation}\label{eq-bound}
\|\mathcal{P}'(u)^{-1}\|_{\mathcal{L}(L^p(Q),W_p(Q))} < 2M.
\end{equation}

Pick  $\rho \in (0,r_0]$  such that $3ML\rho < 1/2$ and  consider
$\mathcal W_\rho \subset \mathcal W_{r_0}.$ 
We will show that for every $u_1\in \mathcal W_\rho$ the Newton iterates $\{u_k\}$ are well-defined and converge quadratically to $u_0.$

Setting   $\epsilon_k := u_k - u_0,$  and suppose initially  that all $u_k\in \mathcal W_\rho.$ Then  we have the expansion
$$
\mathcal{P}(u_k)=\mathcal{P}(u_0+\epsilon_k)=\mathcal{P}(u_0)+\mathcal{P}'(u_0)\epsilon_k+ \mathcal{R}(u_k,u_0)
$$
where the reminder $\mathcal{R}(u_k,u_0)$ satisfies the estimate (cf. \cite{deimling,Z})
\begin{equation}\label{eq-reminder}
\|\mathcal{R}(u_k,u_0)\|_{p,Q}\leq \frac{L}{2}\|\epsilon_k\|^2_{W_p(Q)}
\end{equation}
with 
 $\|\epsilon_k\|_{W_p(Q)}< \rho.$

Transforming the Newton sequences \eqref{eq-NI} we obtain 
\begin{equation}\label{eq-alpha-k}
\begin{split}
\epsilon_{k+1}&=-\mathcal{P}'(u_k)^{-1} (\mathcal{P}(u_k)-\mathcal{P}'(u_k)\epsilon_k)\\
&=-\mathcal{P}'(u_k)^{-1}
\big[(\mathcal{P}'(u_0)-\mathcal{P}'(u_k))\epsilon_k + \mathcal{R}(u_k,u_0)\big],
\end{split}
\end{equation}
since $\mathcal{P}(u_0)=0.$
Applying  \eqref{eq-reminder} along the segment between $u_0$ and $u_k,$ contained in $\mathcal W_\rho,$ and combining this with \eqref{eq-Pprimo},
     \eqref{eq-bound},  and \eqref{eq-alpha-k},  we find
\begin{equation}\label{eq-quadratic}
         \|\epsilon_{k+1}\|_{W_p(Q)} \leq 3ML\|\epsilon_k\|^2_{W_p(Q)}
    \leq (3ML)^{2^k-1}\|\epsilon_1\|^{2^k}_{W_p(Q)}.
\end{equation}
Hence  the convergence is \emph{locally quadratic}.

We now show  that the sequence $\{u_k\}$ indeed  lies     in $\mathcal W_\rho.$
Taking  $u_1\in \mathcal W_\rho,$ 
by  \eqref{eq-quadratic} we get 
$$
\|\epsilon_2\|_{W_p(Q)} <3ML \|\epsilon_1\|^2_{W_p(Q)}< 3ML\rho \|\epsilon_1\|_{W_p(Q)}<\frac12 \|\epsilon_1\|_{W_p(Q)}<\rho
$$
so $u_2\in \mathcal W_\rho.$ By induction, suppose that  $\|\epsilon_j\|_{W_p(Q)}<\rho$ for  all  $j=1,2,\ldots k.$ Then  
\begin{align*}
\|\epsilon_{k+1}\|_{W_p(Q)}&<3ML\rho\|\epsilon_k\|_{L^p(Q)}< \frac12 \|\epsilon_k\|_{W_p(Q)}< \frac{\rho}{2}
\end{align*}
and therefore   $u_{k+1}\in  \mathcal W_\rho,$ that means that  the iteration never  leaves $\mathcal W_\rho.$ 

Finally, from  \eqref{eq-quadratic} we deduce that  the sequence $\{u_k\}$ converges   to $u_0$ in the norm of $W_p(Q),$ namely
$$
\|u_k - u_0\|_{W_p(Q)} \leq (3ML)^{2^{k-1}-1}\|u_1-u_0\|_{W_p(Q)}^{2^k}. 
$$

\end{proof}

\subsection*{Acknowledgment}
The authors are members of INDAM-GNAMPA and the research is partially supported by the project: GNAMPA 2025 - CUP E5324001950001  "Regularity of the solutions of parabolic equations with non-standard degeneracy".
The research of L.Softova is partially
supported by the MIUR PRIN 2022,\\
Grant No.~D53D23005580006 (Salerno Unity).

The research conducted by the authors was partially carried out during their visit to the {\it Institute of Mathematics of the Czech Academy of Sciences}. The authors gratefully acknowledge the staff for their hospitality.

The authors are indebted to the referees for their valuable remarks, which improved the exposition and clarified the proofs.


\begin{thebibliography}{123}


 \bibitem{ANPS1}
 Apushinskaya, D.E.,  Nazarov, A.I., Palagachev, D.K., Softova, L.G.,
{\it Nonstationary Venttsel problems with discontinuous data}, {
J. Differ. Equ.}, {\bf  375} (2023), \mbox{538--566}. 

\bibitem{ANPS3}
Apushkinskaya, D. E., Nazarov, A.I., Palagachev, D.K., Softova, L.G.,
{\it The quasilinear parabolic Venttsel’ problem with discontinuous leading coefficients}, { Funct. Anal. Appl.},  {\bf 57} No.~2 (2023), \mbox{158--163}.




\bibitem{BJW} Byun, S.-S., Oh, J., Wang, L., {\it Global Calderón-Zygmund theory for asymptotically regular nonlinear elliptic and parabolic equations}, { Int. Math. Res. Not.,} {\bf 2015}  No.~17 (2015),  \mbox{8289--8308}. 

\bibitem{BS} Byun, S.-S., Softova, L.,
{\it Asymptotically regular operators in generalized Morrey spaces,} { Bull. Lond. Math. Soc.,} {\bf 52} No.~1 (2020),  \mbox{64--76}.

 \bibitem{CMP} 
Campanato, S.,  {\it A Cordes type condition for nonlinear 
non-variational systems,} {  Rend. Accad., Naz. Sci.
 Detta XL, V. Ser.}, {\bf  13} 
 No.~1 (1989), \mbox{307--321.} 

 \bibitem{CMP0} 
Campanato, S.,
{\it Nonvariational basic parabolic systems of second order}, {
Atti Accad. Naz. Lincei, Cl. Sci. Fis. Mat. Nat., IX. Ser., Rend. Lincei, Mat. Appl.}, {\bf 2} No.~2 (1991), \mbox{129--136.}

 \bibitem{CMP1} 
Campanato, S., {\it  On the condition of nearness between operators},  
 { Ann. Mat. Pura  Appl.}, {\bf 167} (1994), \mbox{243--256.} 
 
\bibitem{CFL} Chiarenza, F., Frasca, M., Longo, P.,  {\it $W^{2,p}$ solvability of the Dirichlet problem for nondivergence form elliptic equations with VMO coefficients}, { Trans. Amer. Math. Soc.}, {\bf 336} (1993), \mbox{841--853}.

\bibitem{deimling}
 Deimling, K., 
\textit{Nonlinear Functional Analysis}, 
Berlin, Springer-Verlag, {\bf XIV} (1985), \mbox{450 p.} 

\bibitem{DK}
Dong, H., Kim, D.,
{\it On the $L^p$-solvability of higher-order parabolic and elliptic systems with BMO coefficients,} {
Arch. Ration. Mech. Anal.}, {\bf  199} No.~3 (2011), 889--941. 

\bibitem{GRk} Griepentrog, J.A., Recke, L.,  {\it Linear elliptic boundary value problems with non-smooth data: Normal solvability in Sobolev--Campanato spaces,} { Mathem. Nachr.,} {\bf 225} (2001), 
\mbox{39--74}.

\bibitem{GRk1} Griepentrog, J.A., Recke, L.,
{\it Local existence, uniqueness and smooth dependence for nonsmooth quasilinear parabolic problems,} { 
J. Evol. Equ.}, {\bf  10} No.~2 (2010), 341--375.

\bibitem{GrR} Gr\"oger, K., Recke, L., {\it Applications of differential calculus to quasilinear elliptic boundary value problems with non-smooth data}, Nonl. Diff. Equ. Appl. (NoDEA), {\bf  13} No.~3 (2006), 263--285.

\bibitem{Jo}
Jones, P.W., 
{Extension theorems for BMO,}  { 
Indiana Univ. Math. J.}, {\bf 29} (1980), \mbox{41--66}.

\bibitem{KimKr} 
Kim, D.,  Krylov, N.V., 
{\it Parabolic equations with measurable coefficients},  
Potential Anal.,  {\bf 26} No.~4 (2007), \mbox{345--361}.

\bibitem{Kr}  Krylov, N.V., 
 {\it Parabolic and elliptic equations with VMO coefficients,}  { Commun. Partial Differ. Equ.}, {\bf 32}  No.~3 (2007), \mbox{453--475}.

\bibitem{Kr1} Krylov, N.V, 
{\it Parabolic equations with VMO coefficients in Sobolev spaces with mixed norms}, 
{ 
J. Funct. Anal.,} {\bf 250} No.~2 (2007), \mbox{521--558}.


\bibitem{Kr2} Krylov, N.V,  {\it Extending BMO functions in parabolic setting,} 
arXiv:2507.09723.

\bibitem{LSU}  Ladyženskaja, O.A., Solonnikov, V.A., Ural'tseva, N.N., {\it 
Linear and Quasilinear Equations of Parabolic Type,}  { Amer. Math. Soc.}, {\bf 23}, 1968.


\bibitem{N} Nazarov, A.I.,  Ural'tseva, N.N., {\it Convex monotone hulls and estimation of the maximum of a solution of a parabolic equation,}  { Soviet Math.}, {\bf 37} (1987), \mbox{851--859}.



\bibitem{Pl}  Palagachev, D.K., {\it Quasilinear elliptic equations with VMO coefficients,}  { Trans. Amer. Math. Soc.}, {\bf 347} (1995), 
\mbox{2481--2493}.


\bibitem{PRS}
Palagachev, D.K., Recke, L., Softova, L.G.,
{\it Applications of the differential calculus to nonlinear elliptic operators with discontinuous coefficients,}  
Math. Ann.,  {\bf  336} No.~3 (2006), \mbox{617--637}.

\bibitem{Rk} Recke, L., {\it Applications of the implicit function theorem to quasilinear elliptic boundary value problem with non-smooth data}, {\ Comm. Part. Diff. Eq.}, {\bf 20} (1995), \mbox{1457--1479}.

\bibitem{RS}
Recke, L., Softova, L.,
{\it Nonlinear parabolic operators with perturbed coefficients,} 
{ Comm. Mathem.  Appl.}, {\bf 9} No.~3 (2018), \mbox{277--292}.

\bibitem{Re} Rescigno, R.,
{\it A quasilinear Cauchy-Dirichelt problem for parabolic equations with $VMO_x$ coefficients}, { Vestn. St. Petersburg Univ., Math.,}   {\bf 58} No.~1 (2025), \mbox{92--100}.

\bibitem{Sf} Softova, L.G., {\it Quasilinear parabolic operators with discontinuous ingredients}, { Nonlin. Anal., Theory Meth. Appl., Ser. A,} {\bf 52} (2003), 
\mbox{1079--1093}.

\bibitem{Sf1} 
Softova, L.,
{\it An integral estimate for the gradient for a class of nonlinear elliptic equations in the plane,} {
Z. Anal. Anwend.,} {\bf 17} No.~1 (1998),  \mbox{57--66}.


\bibitem{Sf2} 
Softova, L.,
{\it Strong solvability for a class of nonlinear parabolic equations,} {Le Matematiche,} {\bf 52} No.~1 (1997), \mbox{59--70}. 


\bibitem{TRS} Tarsia, A., {\it  Differential equations and implicit function: a generalization of near operators theorem, }{ Topol. Meth. Nonl. Anal.}, {\bf 11} (1998), \mbox{115--133}.  


\bibitem{Z} Zeidler, E., {\it Nonlinear Functional Analysis and its Applications}, Vol. I: {\it Fixed Points Theorems}, Springer, Berlin--Heidelberg--New York, 1993.


\end{thebibliography}
\end{document}